%% file: asadmm.tex
\documentclass[twoside]{article}
\input{format.tex}
\input{required.tex}

\SimpleTitle{Accelerated Stochastic ADMM for \\[1mm] Empirical Risk Minimization}
{
\textsc{Chao Zhang$^1$} \, 
\textsc{Hui Qian$^1$}\thanks{Corresponding author.} \, 
\textsc{Zebang Shen$^1$} \, 
\textsc{Tengfei Zhou$^1$} \,\\
\textsc{Jianya Zhou$^2$} \,
\textsc{Jianying Zhou$^2$}}
{zczju, qianhui, shenzebang, zhoutengfei, zhoujy, zjyhz@zju.edu.cn} 

\begin{document}
\maketitle
\thispagestyle{fancy}

\begin{abstract}
\noindent Alternating Direction Method of Multipliers (ADMM) is a popular method for solving large-scale Machine Learning problems. Stochastic ADMM was proposed to reduce the per iteration computational complexity, which is more suitable for big data problems. Recently, variance reduction techniques have been integrated with stochastic ADMM in order to get a faster convergence rate, such as SAG-ADMM and SVRG-ADMM. However, their convergence rate is still suboptimal w.r.t the smoothness constant. In this paper, we propose an accelerated  stochastic ADMM algorithm with variance reduction, which enjoys a faster convergence than all the existing stochastic ADMM algorithms. We theoretically analyse its convergence rate and show its dependence on the smoothness constant is optimal. We also empirically validate its effectiveness and show its priority over other stochastic ADMM algorithms.
\end{abstract}


\lettrine[nindent=0em,lines=2]{T}he Alternating Direction Method of Multipliers (ADMM),  firstly proposed by \clr{\cite{gabay1976dual,glowinski1975approximation}},  is an efficient and versatile tool, which can always guarantee good performance when handling large-scale data-distributed  or big-data related problems, due to its ability of dealing with the objective functions separately and synchronously.  Recent studies have shown that ADMM has a convergence rate of $O(1/N)$ (where $N$ is the number of iterations) for general convex problems \clr{\cite{monteiro2013iteration,he20121,he2012non}}.

Therefore, ADMM has been widely applied to solve machine learning and medical image processing problems in real world, such as \emph{lasso}, \emph{SVM} \clr{\cite{boyd2011distributed}},\emph{group lasso} \clr{\cite{meier2008group}}, \emph{graph guided SVM} \clr{\cite{ouyang2013stochastic}}, and \emph{top ranking} \clr{\cite{kadkhodaie2015accelerated}}. All these  can be easily casted into an Empirical Risk Minimization(ERM) framework. That is
\begin{equation} \label{eq:P}
\begin{aligned}
& \min_{x,y}
& & \frac{1}{n} \sum_{i=1}^{n} f_i(x) + g(y) \\
& \text{subject to}
& & Ax +By = z,
\end{aligned}
\end{equation}

where $\frac{1}{n} \sum_{i=1}^{n} f_i(x)$  is a loss function, each $f_i(x): \RBB^d \rightarrow \RBB$  is a convex component,
and $g(y): \RBB^m \rightarrow \RBB $ is a convex regularizor.
For example, given data samples $\{(w_i,b_i)\}_{i=1}^n$ where $w_i \in \RBB^d$ and $b_i \in \RBB$,
the group lasso problem can be reformulated as
\begin{equation} \nonumber
\begin{aligned}
& \min_{x,y}
& & \frac{1}{n}\sum_{i=1}^n\frac{1}{2}(w_i^T x - b_i)^2 + \nu\|y\|_1 \\
& \text{subject to}
& & Gx = y,
\end{aligned}
\end{equation}
where $G$ is the matrix encoding the group information.

However, in the ERM problem, when the size of the training dataset is large, the minimization procedure in each iteration can be computationally expensive, since it needs to access the whole training set. Many researchers resorted to the increment learning techniques to address such issue. Stochastic ADMM algorithm were proposed \clr{\cite{wang2012online,ouyang2013stochastic,suzuki2013dual}},
though they only have a suboptimal convergence rate $O(1/\sqrt{N})$. Recently, variance reduction techniques have been integrated with stochastic ADMM in order to get a faster convergence rate, such as SAG-ADMM \clr{\cite{zhong2014fast}} and SVRG-ADMM \clr{\cite{zheng2016fast}}. They all achieve a $O(1/N)$ convergence rate when $f(x)$ is  both convex and smooth, and $g(x)$ is convex.

At the same time, a few researchers tried to accelerate the traditional ADMM method by adding a momentum to the original solvers \clr{\cite{goldfarb2013fast,kadkhodaie2015accelerated}}. An important work is the AL-ADMM algorithm proposed by \clr{\cite{ouyang2015accelerated}}, which improves the rate of convergence from $O(1/N)$ to $O(1/N^2)$ in terms of its dependence on the smoothness constant $L_f$ of $f(x)$.

Inspired by these two lines of research, we propose an accelerated stochastic ADMM algaorithm called \emph{Accelerated SVRG-ADMM} (ASVRG-ADMM), which incorporates both the variance reduction and the acceleration technique together. By utilizing a set of auxiliary points, we speed up the original SVRG-ADMM algorithm without introducing much extra computation. We  theoretically analyze the convergence rate of ASVRG-ADMM and show it is $O(1/N^2)$ for the smoothness constant \footnote{The overall convergence is $O(\frac{aL_f}{N^2}+\frac{b}{N})$, where $aL_f$ will dominate $b$ in most ERM problems, since big dataset means big smoothness constant.}. Experimental results also show the proposed algorithm outperforms other stochastic ADMM algorithms in the big data settings.

\section{Notation and Preliminaries}
For a vector $x$, $\|x\|$ denotes the $\ell_2$-norm of $x$, and $\|x\|_1$ denotes the $\ell_1$-norm. For a matrix $X$, $\|X\|_2$ denotes its spectral norm. For a random variable $a$, we use $\EBB$($a$) to denote its expectation and $\VBB$($a$) $= \EBB^2(\|a\|) - \EBB(\|a\|^2)$ to denote its variance. For a function $h(x)$, we denote its gradient and subgradient as $\nabla h$ and $\partial h(x)$ respectively.  A function $h(x)$ is $L$-smooth if it is differentiable and $\|h'(x)- h'(y)\| \leq L\|x-y\|$, or equivalently,
\[
h(x) \le h(y) +\langle \nabla h(y), x-y\rangle + \frac{L}{2} \| x-y \|^2.
\]
We call $L$ the \emph{smoothness constant}.

For the convenience of notation, we denote $f(x) = \frac{1}{n}\sum_{i=1}^n f_i(x)$, $u=(x,y)$, and $w = (x,y,\lambda)$. We assume that the effective domain $\XM$ of $x$ and $\YM$ of $y$ are bounded, that is
\begin{align*}
D_{x^*} &= \sup_{x \in \XM} \|x-  x^*\|, \\
D_{y^*,B} &=\sup_{y \in \YM} \|B(y-y_*)\|, \textrm{and}\\
D_{X,A} &= \sup_{x_a,x_b \in \XM}\|A(x_a-  x_b)\|
\end{align*}
exists and not equal to infinity, where $A$ and $B$ are matrices in the constraint of (\ref{eq:P}). We also assume that the optimal solution $u^*=(x^*,y^*)$ of problem (\ref{eq:P}) and the optimal of each $f_i$ exists.

Similar to those in \clr{\cite{ouyang2015accelerated}}, we denote the gap function as follows.
\begin{dfn}[Gap Function]
For any $w = (x,y,\lambda)$ and $\bar{w} = (\bar{x},\bar{y},\bar{\lambda})$, we define
\begin{align*}
Q(\bar{x},\bar{y},\bar{\lambda};x,y,\lambda)&=
[f(x)+g(y)+\langle \bar{\lambda},Ax + By -c\rangle]
\\
&- [f(\bar{x})+g(\bar{y})+\langle \lambda,A\bar{x} + B\bar{y} -c\rangle].
\end{align*}
\end{dfn}
For the simplicity of notation, we also write $Q(\bar{w};w)=Q(\bar{x},\bar{y},\bar{\lambda};x,y,\lambda)$.

\section{Related Work}
Here we focus on the constraint optimization problem
\begin{equation}\label{eq:form}
\min_{x,y} f(x) + g(y), \quad \textrm{ s.t.} \quad  Ax + By =z,
\end{equation}
where $f(x)$ is both convex and  $L_f-$smooth, and $g(y)$ is convex. Besides, we also assume that each $f_i$ is $L_i$- smooth.

To solve  problem (\ref{eq:form}), ADMM starts with the augmented Lagrangian of the original problem:
\begin{align*}
 L_{\beta} (x,y,\lambda) &:= f(x) + g(y) -\langle\lambda,Ax+By-z\rangle \\
 &+ \frac{\beta}{2}\|Ax+By-z\|^2,
\end{align*}
where $\beta > 0$ is a constant, and $\lambda $ is the vector of Lagrangian multipliers. In each round, ADMM minimizes $L_\beta$ with respect to the variables $x$ and $y$ alternatively given the other fixed, followed by an update of the vector $\lambda$.
Specifically, it updates $x, y$ and $\lambda$ as follows.
\begin{align*}
x_{t} &= \textrm{argmin} L_\beta(x,y_{t-1},\lambda_{t-1}),\\
y_{t} &= \textrm{argmin} L_\beta(x_{t},y,\lambda_{t-1}),\\
\lambda_{t}  &= \lambda_{t} - \beta(A x_{t} + B y_{t} - c).
\end{align*}

In order to avoid the access to the whole dataset in the $x$ update and further
reduce the per-iteration computation complexity,
stochastic ADMM algorithm(SADMM) was proposed by \clr{\cite{ouyang2013stochastic}}.
However, their method can only achieve a suboptimal convergence rate $O(\frac{L_f D_{x^*}^2}{N}+\frac{c}{N}+\frac{D_{x^*}\sigma}{\sqrt{N}})$,
where $c$ is a constant that depends on $D_{x^*},D_{y^*,B}$ and $D_{X,A}$, and  $\sigma$ is the upper bound of the variance of stochastic gradients.
Typically, $\sigma$ and $L_fD_{x^*}^2$ can be rather large and $c$ is small compared to them.
Based on \clr{\cite{ouyang2013stochastic}},
\clr{\cite{azadi2014towards}} proposed an accelerated stochastic ADMM algorithm(ASADMM) to reduce the dependence of $L_f$,
which has a $O(\frac{L_f D_{x^*}^2}{N^2}+\frac{c}{N}+\frac{D_{x^*}\sigma}{\sqrt{N}})$ convergence rate.
The $O(\frac{L_f D_{x^*}^2}{N^2})$ is optimal w.r.t. the smoothness constant $L_f$,
but due to the high variance of stochastic gradients, their method is far away from optimal.

Borrowed idea from variance reduction techniques used in stochastic gradient descent literature,
SAG-ADMM \clr{\cite{zhong2014fast}} and SVRG-ADMM \clr{\cite{zheng2016fast}} are recently proposed.
These methods enjoy similar convergence rate $O(\frac{L_fD_{x^*}^2}{N}+\frac{c_1}{N})$, where $c_1$ is a constant which depends on $D_{x^*}$, $D_{y^*}$ , $D_{X,A}$,$\|A\|_2$ and $\|B\|_2$, and can be rather small compared to $L_fD_{x^*}^2$ (different algorithms may have different $c_1$).

Recently, acceleration technique has been integrated with variance reduction in stochastic gradient descent literature \clr{\cite{hien2016accelerated,allen2016katyusha}},
which improves the convergence rate from $O(1/N)$ to $O(1/N^2)$,
but it is still not clear how to combine this two techniques in stochastic ADMM literature to get a more effective method.

\section{ASVRG-ADMM Algorithm}
In this section, we first propose an general ASVRG-ADMM
framework for solving (\ref{eq:form}), and then analyze relation between the convergence rate and the parameter settings.
Finally, we conclude our Accelerated SVRG-ADMM algorithm by specifying the particular parameter setting.
The proposed framework is presented in Algorithm \ref{alg:asvrgadmm}.

\IncMargin{1em}

\begin{algorithm}[H]\label{alg:asvrgadmm}
\BlankLine
Initialization : $\tilde{x}_{0}=x_{m,0},\tilde{y}_0 =y_{m,0},\tilde{\lambda}_0 =\lambda_{m,0} $  such that $A\tilde{x}_0 + B \tilde{y}_0 -c =0$

    \For{$ s= 1,2\cdots N$}
    {
        Update $\alpha_{1,s},\alpha_{2,s},\alpha_{3,s},\theta_{s}$ and $\rho_{s}$\;
        $x_{0,s} = x_{m,s-1}$;
        $y_{0,s} = y_{m,s-1}$;
        $\lambda_{0,s} = \lambda_{m,s-1}$\;
        $x_{0,s} ^{ag}= x_{m,s-1}^{ag}$;
        $y_{0,s}^{ag} = y_{m,s-1}^{ag}$;
        $\lambda_{0,s}^{ag} = \lambda_{m,s-1}^{ag}$\;
        $\tilde{v}_s = \frac{1}{n}\sum_{i=1}^n \nabla f_i(\tilde{x}_{s-1})$\;

        \For{$t = 1,2,\cdots,m$}
        {
            $x_{t,s}^{md} = \alpha_{1,s} x_{t-1,s}^{ag} + \alpha_{2,s} x_{t-1,s} + \alpha_{3,s} \tilde{x}_{s-1}$\;
            Sample $i_t$ uniformly from $\{1,2,\cdots,n\}$\;
            $v_{t,s} = \nabla f_{i_t}(x_{t,s}^{md}) -\nabla f_{i_t}(\tilde{x}_{s-1}) + \tilde{v}_s$\;
           $x_{t,s} = \textrm{argmin}_x \bar L_{t,s}(x,y_{t-1,s},\lambda_{t-1,s},\chi)$\;
            $x_{t,s}^{ag} = \alpha_{1,s} x_{t-1,s}^{ag} + \alpha_{2,s} x_{t,s} + \alpha_{3,s} \tilde{x}_{s-1}$\;
            $y_{t,s} = \textrm{argmin}_y  g(y) + \langle \lambda_{t-1,s}, A x_{t,s} + By -c \rangle+\frac{\theta_{s}}{2}\|A x_{t,s} + B y -c\|^2$\;
            $y_{t,s}^{ag} = \alpha_{1,s} y_{t-1,s}^{ag} + \alpha_{2,s} y_{t,s} + \alpha_{3,s} \tilde{y}_{s-1}$\;
            $\lambda_{t,s} = \lambda_{t-1,s} + \rho_{s} (A x_{t,s} + B y_{t,s} -c)$\;
            $\lambda_{t,s}^{ag} = \alpha_{1,s} \lambda_{t-1,s}^{ag} + \alpha_{2,s} \lambda_{t,s} + \alpha_{3,s} \tilde{\lambda}_{s-1}$\;
        }
        $\tilde{x}_s = \frac{1}{m} \sum_{t=1}^m x_{t,s}^{ag}$;
        $\tilde{y}_s = \frac{1}{m} \sum_{t=1}^m y_{t,s}^{ag}$\;
        $\tilde{\lambda}_s = \frac{1}{m} \sum_{t=1}^m \lambda_{t,s}^{ag}$\;

    }
    \BlankLine
\SetKwInOut{Output}{Output}
\Output {$\hat{w}_N = \frac{1}{1+\alpha_{3,N+1}m}w_{m,N}^{ag}+\frac{\alpha_{3,N+1}m}{1+\alpha_{3,N+1}m}\tilde{w}_{N}$
}
\caption{ASVRG-ADMM Framework}
\end{algorithm}
\DecMargin{1em}

In this framework,
the updates of $x$ and $y$ are actually alternatingly minimizing the weighted linearized augmented Lagrangian:

\begin{align*}
\bar{L}_{t,s} (x,y,\lambda_{t,s},\chi) &= f(x_{t-1,s})\\
&+\langle v_{t,s} ,x\rangle +g(y)+
\langle \lambda_{t-1,s} ,Ax +By-c\rangle\\
 &+\chi \theta_{s}  \langle A x ,A x _{t-1,s} + B y_{t-1,s} -c\rangle
\\& +\frac{(1-\chi)\theta_{s}}{2} \|A x + B y_{t-1,s} -c\|^2\\
&+\frac{\eta_{s}}{2}\|x - x_{t-1,s}\|^2,
\end{align*}
where $f(x)$ in the original augmented lagrangian $L_{\beta}(x,y,\lambda)$ is linearized as $ f(x_{t-1,s})+\langle v_{t,s} ,x\rangle +\frac{\eta_{s}}{2}\|x - x_{t-1,s}\|^2_2$ and a weight parameter $\theta_s$ is added to the constraint measure $\frac{\|Ax+By-c\|^2_2}{2}$.
However, we must notice that $v_{t,s}$ is an unbiased estimation of $\nabla f(x_{t,s}^{md})$ instead of $\nabla f(x_{t,s})$.
Here,$\chi$ is an indicator variable that is either 0 or 1.
If $\chi = 0$, augmented term $\|Ax+By_{t-1,s}-c\|^2$ is preserved ,
and the update in line (10) of Algorithm~\ref{alg:asvrgadmm} is

\begin{align}
x_{t,s} &= (\eta_s I + \theta_s A^T A)^{-1}(\eta_s x_{t-1,s} - v_{t,s}
\nonumber \nonumber\\
&-A^T\lambda_{t-1,s}-\theta_{s}A^T(By_{t-1,s}-c))\label{xupdate:non}.
\end{align}
Since every update in (\ref{xupdate:non}) involves a matrix inverse,
which can be quite complicated when $A$ is large,
we can set $\chi = 1$ to linearize this part to simplify the computation.
When $\chi = 1$,
the augmented term $\|Ax+By-c\|^2$ in $\bar{L}_{t,s}(x,y,\lambda_{t,s})$ is  linearized as $\langle Ax_{t-1,s} + By_{t-1,s}-c, Ax \rangle$ in the $t$-th iteration,
and the update in line (10) of Algorithm~\ref{alg:asvrgadmm} is simplified as

\begin{align}
x_{t,s} &= x_{t-1,s} - \frac{1}{\eta_s} (A^T\lambda_{t-1.s}+v_{t,s}
\nonumber\\
&+\theta_{s}A^T(Ax_{t-1,s}+By_{t-1,s}-c))\label{xupdate:lin}.
\end{align}

Similar to other accelerated methods,
we construct a set of auxiliary sequences  $\{x^{md}\}$ , $\{x^{ag}\}$ ,$\{y^{ag}\}$ and $\{\lambda^{ag}\}$ .
Here the superscript "ag" stands for "aggregate",
and "md" stand for "middle".
We can see that all $\{x_{t,s}^{md}\}$ and $\{x_{t,s}^{ag}\}$ are actually weighted sums of all the previous $\{x_{i,k}\}$ and all $\{y_{t,s}^{ag}\}$
and $\{\lambda_{t,s}^{ag}\}$ are weighted sums of previous $\{y_{i,k}\}$ and $\{\lambda_{i,k}\}$,respectively.
We require all the weight parameters $\alpha_{1,s},\alpha_{2,s},\alpha_{3,s}$ belong to $(0,1)$
and their sum equals 1 to make all the auxiliary points to be a convex combination of three previous point.
If the wights $\alpha_{2,s} = 1$ for all $s$,
then $x_{t,s}^{md} = x_{t-1,s}$,
and the aggregate points $x_{t,s}^{ag} = x_{t,s}$ ,
$y_{t,s}^{ag} = y_{t,s}$ and $\lambda_{t,s}^{ag} = \lambda_{t,s}$.
In this situation,
if we set all $\theta_{s} = \rho_{s} = \beta$,
the ASVRG-ADMM becomes SVRG-ADMM.
However, with carefully chosen parameters
we can significantly improve the rate of convergence.

\subsection{Convergence Analysis}
In this subsection, we will give the convergence analysis of the proposed framework, and conclude our ASVRG-ADMM algorithm.

Before we address the main theorem on the convergence rate, we will firstly try to bound the variance of the stochastic gradient.
In order to control the variance of the stochastic gradient,
we first take a snapshot $\tilde{x}_{s-1}$ at the beginning of each outer iteration and calculate its full gradient $\tilde{v}_s$,
and then randomly sample one $f_{i_t}(x)$ and update $v_{t,s} = \nabla f_{i_t}(x_{t,s}^{md}) -\nabla f_{i_t}(\tilde{x}_{s-1}) + \tilde{v}_s$ in each inner loop.
Similar to \clr{\cite{allen2016katyusha}} and \clr{\cite{hien2016accelerated}},
we use a different upper bound from the analysis of SVRG-ADMM.
\begin{lemma}\label{lemma:variance}
In each of the inner iteration,the variance of $v_{t,s}$ is bounded by
\begin{eqnarray*}
\EBB \|v_{t,s} - \nabla f(x_{t,s}^{md})\|^2_2
&\le&
2L_Q(
f(\tilde{x}_{s-1}) - f(x_{t,s}^{md}) 
\\
&-&\langle \nabla f(x_{t-1,s}^{md}),\tilde{x}_{s-1}-x_{t,s}^{md}\rangle),
\end{eqnarray*}
where $L_Q = max_i\{L_i\}$.
\end{lemma}
\begin{proof}
\begin{align*}
&\EBB \|v_{t,s} - \nabla f(x_{t,s}^{md})\|^2_2
\\&=
\EBB \|\nabla f_{i_t}(x_{t,s}^{md}) -\nabla f_{i_t}(\tilde{x}_{s-1}) + \nabla f(\tilde{x}_{s-1}) - \nabla f(x_{t,s}^{md})\|^2_2
\\
&\le\EBB \|\nabla f_{i_t}(x_{t,s}^{md}) -\nabla f_{i_t}(\tilde{x}_{s-1})\|^2_2
\\
&\le\frac{1}{n}\sum_{i=1}^n 2L_i(f_{i}(\tilde{x}_{s-1})-f_i(x_{t,s}^{md})-\langle \nabla f_i(x_{t,s}^{md}),\tilde{x}_{s-1}-x_{t,s}^{md}\rangle)
\\
&\le2L_Q(f(\tilde{x}_{s-1}) - f(x_{t,s}^{md}) - \langle \nabla f(x_{t-1,s}^{md}),\tilde{x}_{s-1}-x_{t,s}^{md}\rangle).
\end{align*}
The first equality is according to the definition of $v_{t,s}$.
In the first inequality,we use $\EBB \|a - \EBB (a)\|^2_2 \le \EBB\|a\|^2_2$.
The second inequality is because of the smoothness of each $f_i(x)$,
and the last one is just due to the definition of $L_Q$.
\end{proof}

Utilizing the above lemma, we are able to obtain an upper bound of the progress of the gap funciton $Q$ in each inner loop.
\begin{lemma}\label{lemma:inneriteration}
 In Algorithm \ref{alg:asvrgadmm}, if we choose $\theta_{s}\ge\rho_{s}$ and $\eta_{s}\ge\bar{L}_s\alpha_{2,s}+\chi\theta_{s}\|A\|^2_2$,
we have

\begin{align*}
\EBB Q(x^*,y^*,\lambda;w_{t,s}^{ag}) &-
\alpha_{1,s}Q(x^*,y^*,\lambda;w_{t-1,s}^{ag}) \\
&-\alpha_{3,s}Q(x^*,y^*,\lambda;\tilde{w}_{s-1})\\
&\le \alpha_{2,s,t}[\frac{\eta_{s}}{2}(\|x_{t,s}-x^*\|^2-\|x_{t,s}-x^*\|^2)
\\
&+\frac{1}{2\rho_{s}}(\|\lambda_{t-1,s}-\lambda\|^2-\|\lambda_{t,s}-\lambda\|^2)
\nonumber\\
&+ \frac{\chi\theta_{s}}{2} \|A(x_{t,s}-x^*)\|^2 \\
&-\frac{\chi\theta_{s}}{2}\|A(x_{t-1,s}-x^*)\|^2
\nonumber\\
&+\frac{\theta_{s}}{2}\|Ax^*+By_{t-1,s}-c\|^2\\
&-\frac{\theta_{s}}{2}\|Ax^*+By_{t,s}-c\|^2],
\nonumber
\end{align*}
where $\bar{L}_{s}\ge L_Q/\alpha_{3,s}+L_f$ and the expectation is w.r.t $v_{t,s}$.
\end{lemma}
The proof of this inequality is lengthy and tedious,
so we leave it in the supplement.
We can see that this bound is not only related to $x_{t,s}$ and $x_{t-1,s}$, but also related to $\tilde{x}_{s-1}$.
It is very common in snapshot based algorithms.
Lemma \ref{lemma:inneriteration} proves crucial for doing the induction to obtain the next step towards our final conclusion.

\begin{lemma}\label{lemma:final}
If in the $(s+1)$-th outer iteration, we choose
$\frac{1-\alpha_{2,s+1}}{\alpha_{2,s+1}^2} =\frac{1}{\alpha_{2,s}^2}$ and $\frac{\alpha_{3,s+1}}{\alpha_{2,s+1}^2}
=\frac{1-\alpha_{1,s}}{\alpha^2_{2,s}}$
,and make $\theta_{s}=\beta_1\alpha_{2,s}$, $\rho_{s}=\frac{\beta_2}{\alpha_{2,s}}$ ,$\eta_{s}=(\bar{L}+\chi\beta_1\|A\|_2^2)\alpha_{2,s}$.
Then in the $N$-th iteration, we have
\begin{align}
&\EBB [f(\hat{x}_{N})+g(\hat{y}_{N}) - f(x^*) - g(y^*)+\gamma\|A\hat{x}_N+B\hat{y}_N-c\|^2]
\nonumber\\
&\le\frac{\alpha_{2,N+1}^2}{1+m\alpha_{3,N+1}}[\frac{1+m\alpha_{3,1}}{\alpha_{2,1}^2}(f(\tilde{x}_0)+g(\tilde{y}_0) - f(x^*) - g(y^*))
\nonumber\\
&+\frac{\bar{L}+\chi\beta_1\|A\|_2^2}{2}D^2_{x^*}+\frac{1}{2\beta_2}\gamma^2
+\frac{\chi\beta_1}{2}D^2_{A,X}+\frac{\beta_1}{2}D^2_{y^*,B}],
\nonumber
\end{align}
here we need choose $\beta_1$ and $\beta_2$ to satisfy $\theta_{s} \ge \rho_{s}$, and $\bar{L} \ge \max_s\{L_Q/\alpha_{3,s}  + L_f\}$.
\end{lemma}

Before we conclude our main theorem of the convergence rate of ASVRG-ADMM,
we need to have a look at all the parameter constraints in Lemma ~\ref{lemma:final}.
\begin{displaymath}
\left\{
\begin{array}{l}
\frac{1}{\alpha_{2,s}^2} = \frac{1-\alpha_{2,s+1}}{\alpha_{2,s+1}^2},
\frac{1-\alpha_{1,s}}{\alpha_{2,s}} = \frac{\alpha_{3,s+1}}{\alpha_{2,s+1}^2};
\\
\alpha_{1,s}+\alpha_{2,s}+\alpha_{3,s} = 1
\\
\theta_{t,s} = \beta_{1} \alpha_{2,s},
\rho_{t,s} = \frac{\beta_{2}}{\alpha_2,s},
\theta_{s} \ge \rho_{s};
\\
\eta_{t,s} = \bar{L}\alpha_{2,s} + \chi \|A\|^2_2 \theta_{s}, \quad
\bar{L} \ge \max_s\{L_Q/\alpha_{3,s}  + L_f\}.
\end{array}
\right.
\end{displaymath}

 To satisfy the constraints on weight parameters $\alpha_{1,s},\alpha_{2,s},\alpha_{3,s}$,
 we can actually calculate the update rule of each
 \begin{align*}
 \alpha_{1,s+1}& = \alpha_{1,s}*(1-\alpha_{2,s+1}),
 \\
 \alpha_{2,s+1} &= \frac{\sqrt{\alpha_{2,s}^4+4\alpha_{2,s}^2}-\alpha_{2,s}^2}{2},
 \\
\alpha_{3,s+1}& = (1-\alpha_{1,s})*(1-\alpha_{2,s+1}).
\end{align*}
The following lemma depicts the property of this sequence.

\begin{lemma}\label{lemma:parameter}
If $\alpha_{1,s}\in (0,1),\alpha_{2,s} \in(0, \frac{2}{2+s}],\alpha_{3,s}\in(0,1),$ and their sum equals to 1,
 then  $\alpha_{1,s+1}\in(0,\alpha_{1,s})$, $\alpha_{2,s+1}\in
 (0,\min\{\alpha_{2,s},\frac{2}{s+3}\})$, $\alpha_{3,s+1}\in (\alpha_{3,s},1)$ and their sum equals to 1,too.
\end{lemma}
\begin{proof}
Denote $h(a) = \frac{\sqrt{a^4+4a^2}-a^2}{2}$ as a function of $a$,
we have $\nabla h(a) >0$ for $a\in (0,1)$.
Then $\alpha_{1,s+1} \le h(\frac{2}{s+2})$.
Since $h(\frac{2}{s+2}) < \frac{2}{s+3}$,
then $\alpha_{1,s+1} < \frac{2}{s+3}$.
Besides, it is easy to verify $\alpha_{2,s+1} - \alpha_{2,s} \le 0$ and
$0<\alpha_{2,s+1}<1$.
then we have $\alpha_{2,s+1} \in (0,\min\{\alpha_{2,s},\frac{2}{s+3}\})$.
Since $\alpha_{1,s+1} = \alpha_{1,s}*(1-\alpha_{2,s+1})$ ,
then $\alpha_{1,s+1} \in (0,\alpha_{1,s})$.
As $\alpha_{3,s+1} -\alpha_{3,s} = \alpha_{1,s} +\alpha_{2,s}-
\alpha_{1,s+1}-\alpha_{2,s+1}$, then $\alpha_{3,s+1} \in (\alpha_{3,s},1)$.
\end{proof}

This means $\{\alpha_{1,s}\}$ and $\{\alpha_{2,s}\}$ are decreasing and $\{\alpha_{3,s}\}$ is increasing.
Besides, we also have all $\alpha_{2,s} \le \frac{2}{s+2}$ if $\alpha_{2,1} \le \frac{2}{3}$.
Actually,
we can verify that if $\alpha_{2,1} = 2/3$,
then $\alpha_{2,s} \to \frac{2}{2+s}$ as $s \to \infty$ ,
and  $\beta_1 = N$, $\beta_2 = \frac{1}{N}$  will ensure $\theta_{s} \ge 1 \ge \rho_{s}$.
Since $\alpha_{3,s}$ is increasing,
$\bar{L} \ge L_Q/\alpha_{3,1}+L_f$ will satisfy the constraint.
Based on all these,
we are now ready to present our main theorem of convergence rate:

\begin{them}\label{them:main}
If we initialize $\alpha_{2,1} = \frac{2}{3}$,
$\alpha_{3,1} \in (0,\frac{1}{3})$,
$\alpha_{1,1} = 1-\alpha_{2,1}-\alpha_{3,1}$,
then with proper parameter setting the same as in lemma~\ref{lemma:final} and $\beta_1 = N$,
$\beta_2 = \frac{1}{N}$,
$\bar{L} = L_Q/\alpha_{3,1}+L_f$,
we have
\begin{align*}
&\EBB [f(\hat{x}_{N})+g(\hat{y}_{N}) - f(x^*) - g(y^*)+\gamma\|A\hat{x}_N+B\hat{y}_N-c\|^2]
\nonumber\\
&=\OO\big(\frac{1}{N^2}(f(\tilde{x}_0)+g(\tilde{y}_0) - f(x^*) - g(y^*))
+\frac{L_Q+L_f}{mN^2}D^2_{x^*}
\nonumber\\
&+\frac{1}{mN}
[\frac{\chi\|A\|_2^2}{2}D^2_{x^*}+\gamma^2
+\chi D^2_{A,X}+D^2_{y^*,B}]\big).
\end{align*}
\end{them}
\begin{proof}
%
%
Since $\alpha_{2,N+1}\le\frac{2}{N+3}$ and $\frac{1}{1+m\alpha_{3,N+1}} \le \frac{1}{1+m\alpha_{3,1}}$,
then according to \ref{lemma:final},
we have
\begin{align*}
&f(\hat{x}_{N})+g(\hat{y}_{N}) - f(x^*) - g(y^*)+\gamma\|A\hat{x}_N+B\hat{y}_N-c\|^2
\nonumber\\
&\le\frac{4}{\alpha_{2,1}^2(N+3)^2}(f(\tilde{x}_0)+g(\tilde{y}_0) - f(x^*) - g(y^*))
\\
&+\frac{N}{2m(N+3)^2}
[{\chi\|A\|_2^2}D^2_{x^*}+\gamma^2
+\chi D^2_{A,X}+D^2_{y^*,B}]
\\
&+\frac{2(L_Q/\alpha_{3,1}+L_f)}{(N+3)^2(1+m\alpha_{3,1)}}D^2_{x^*}
\nonumber\\
&=\OO\big(\frac{1}{N^2}(f(\tilde{x}_0)+g(\tilde{y}_0) - f(x^*) - g(y^*))
+\frac{L_Q+L_f}{mN^2}D^2_{x^*}
\nonumber\\
&+\frac{1}{mN}
[\chi\|A\|_2^2D^2_{x^*}+\gamma^2
+\chi D^2_{A,X}+D^2_{y^*,B}]\big).
\end{align*}
\end{proof}

The convergence result in {\bf Theorem \ref{them:main}} mainly consists of the convergence of three parts:
\begin{itemize}
\item $f(\tilde{x}_0)+g(\tilde{y}_0) - f(x^*) - g(y^*)$:
    This part measures the influence of  the initial objective value on the convergence rate.
    Actually in most ERM problems such as (group/fused) lasso,
    $l_1/l_2$ penalized regresion/logistic regression,
    $f(\tilde{x}_0)+g(\tilde{y}_0) - f(x^*) - g(y^*)$ can be quite small if we simply initialized $\tilde{x}_0 = 0$ and $\tilde{y}_0 =0$.
    For example,
    in two label classification,
    for any commonly used regularizer $g(y)$($l_1/l_2$ norm;group/fussed lasso penalty),
    if $f(x)$ is the logistic loss then
    $f(0)+g(0) = In(2) $;
    if $f(x)$ is the square loss then $f(0) + g(0) = \frac{1}{2}$.
    Since $f(x_*)+g(y_*)$ here must be great than 0,
    then this part can be really small compared to the other two parts.
\item
    $(L_Q+L_f)D_{x^*}^2$:
    This is the dominant part of the convergence since $L_Q$ and $L_f$ can be rather large for ERM problems with a big dataset,
    and we assume $L_Q = O(L_f)$ without lost of generality.
\item
    $\chi\|A\|_2^2D^2_{x^*}+\gamma^2+\chi D^2_{A,X}+D^2_{y^*,B}$: Here $D_{x^*}$, $D_{A,X}$, $D_{y^*,B}$ are constant denoting the boundary of the effective domain of $x$ and $y$.They are always assumed to be finite in the analysis of stochastic ADMM algorithms and are typically small compared to $(L_Q+L_f)D_{x^*}^2$.
    It also shows the linearization of $\|Ax+By-c\|^2$  ($\chi = 1$) will not effect the convergence significantly.


\end{itemize}

Thus if we make the count of iter loop as $m = n$, then the convergence rate of ASVRG-ADMM
can be written as $O(\frac{L_fD_X^2}{nN^2}+ \frac{c_3}{N^2} +\frac{c_2}{nN})$,
where $c_2$ and $c_3$ are rather small compared to $L_fD_X^2$.
If we use similar notation,the convergence rate of SAG-ADMM is $O(\frac{L_fD_X^2}{N}+\frac{c_4}{nN})$ and the convergence rate of SVRG-ADMM is $O(\frac{L_fD_X^2}{nN}+ \frac{c_5}{N} +\frac{c_6}{nN})$,
where $c_4,c_5,c_6$ are all small compared to $L_fD_X^2$ in most of the ERM problems.
Typically, $c_3,c_5$ are of the same order, $c_2,c_4,c_6$ are of the same order and $c_3,c_5$ are  much smaller than $c_2,c_4,c_6$.
It is clear from the above analysis that our proposed accelerated stochastic ADMM algorithm(ASVRG-ADMM) are much better than other
stochastic variance reduction ADMM without acceleration in dealing with ERM problem.
Moreover, when the dataset is large and $f(x)$ has a large smoothness constant, the second part will dominant the convergence
and the convergence rate will be optimal $O(1/N^2)$ w.r.t. the smoothness constant,
while other methods only achieve $O(1/N)$ convergence.
As for the vanilla stochastic ADMM algorithms without variance reduction SADMM \clr{\cite{ouyang2013stochastic}} and ASADMM \clr{ \cite{azadi2014towards}}, both of their analyses include the variance of the stochastic gradient $\sigma$ explicitly, and the convergence rate is $O(\frac{\sigma D_X}{\sqrt{nN}})$  w.r.t. $\sigma$,
which is uncontrollable and can be extremely bad in practice.

\subsection{Discussion on the parameters}
Let's have a deeper insight into the parameters in the ASVRG-ADMM algorithm.
The weight parameters $\alpha_{1,s},\alpha_{2,s} \alpha_{3,s}$ are updated once in each out iteration,
and stay the same in the inner loop.
In each inner loop,
the update of  $x_{t,s}^{md}$ and $x_{t,s}^{ag}$ are:
\[
x_{t,s}^{md} = \alpha_{1,s} x_{t-1,s}^{ag} + \alpha_{2,s}x_{t-1,s}+\alpha_{3,s}\tilde{x}_{s-1};
\]
\[
x_{t,s}^{ag} = \alpha_{1,s} x_{t-1,s}^{ag} + \alpha_{2,s}x_{t,s}+\alpha_{3,s}\tilde{x}_{s-1}.
\]
It seems that both $x_{t,s}^{md}$ and $x_{t,s}^{ag}$ are somehow trapped at $\tilde{x}^{s-1}$,
and we even gradually increase the weight on $\tilde{x}_{s-1}$.
However,
it is not incomprehensible since we are actually analyze the convergence of sequences $\{\tilde{x}_{s}\}$
(since $\alpha_{3,N+1}m>> 1$, $\bar{w}_s \approx \tilde{w}_s$).
According to line 18 in Algorithm ~\ref{alg:asvrgadmm}
\[
\tilde{x}_s = \frac{\alpha_{1,s}}{m}\sum_{i=0}^{m-1}x_{i,s}^{ag} +\frac{\alpha_{2,s}}{m}\sum_{i=1}^m x_{i,s} + \alpha_{3,s}\tilde{x}_{s-1},
\]
and it is updated after we use one full gradient and n stochastic gradient.
While in the deterministic accelerated ADMM algorithm proposed by \clr{\cite{ouyang2015accelerated}},
the update of the auxiliary points are \[
x_{s}^{md} = \alpha_{1,s} x_{s-1}^{ag} + \alpha_{2,s}x_{s-1};\quad
x_{s}^{ag} = \alpha_{1,s} x_{s-1}^{ag} + \alpha_{2,s}x_{s}.
\]
Here $x_s$ is updated based on the full gradient,
and $\alpha_{2,s}$ is decreased similar to $\alpha_{2,s}$ in our algorithm.
Since $x_{s}^{ag}$ is updated based an one full gradient
and the convergence of their algorithm is w.r.t. $x_{s}^{ag}$,
$x_{s}^{ag}$ here is similar to $\tilde{x}_{s}$   instead of $x_{t,s}^{ag}$ in our ASVRG-ADMM algorithm.
As the weight of previous $x_{s-1}^{ag}$  is also gradually increased,
it is not counterintuitive for us to gradually increase the weight of $\tilde{x}_{s-1}$.
Actually, when $m = 1$, ASVRG-ADMM is actually the same as
the deterministic accelerated ADMM method proposed by \clr{\cite{ouyang2015accelerated}}.

\begin{table*}[t]{\label{table:data}}
\caption{A summary of the datasets.}
\label{tab:datasets}
\begin{center}
\begin{footnotesize}
\begin{tabular}{c c c c c c}
\hline
	{\bf Dataset}      & ~~w8a~~~&~~~ijcnn1~~&~~a9a~~&~~covtype~~\\
\hline
    {\bf \#Instance}   &  $49,749$ &$49,990$ &   $32,561$  &$581012$ \\
    {\bf \#Attribute~}&   $300$   & $22$    &   $123$    &$54$      \\
    {\bf \#$L_f$} &$2.6448$&$0.2305$&$6.2877$&$2.0164*10^7$\\
\hline
\end{tabular}
\end{footnotesize}
\end{center}
\end{table*}

\begin{figure*}[!ht]\label{fig:general}
\begin{center}
\centering
\subfigure[a9a]{\includegraphics[width = 40mm]{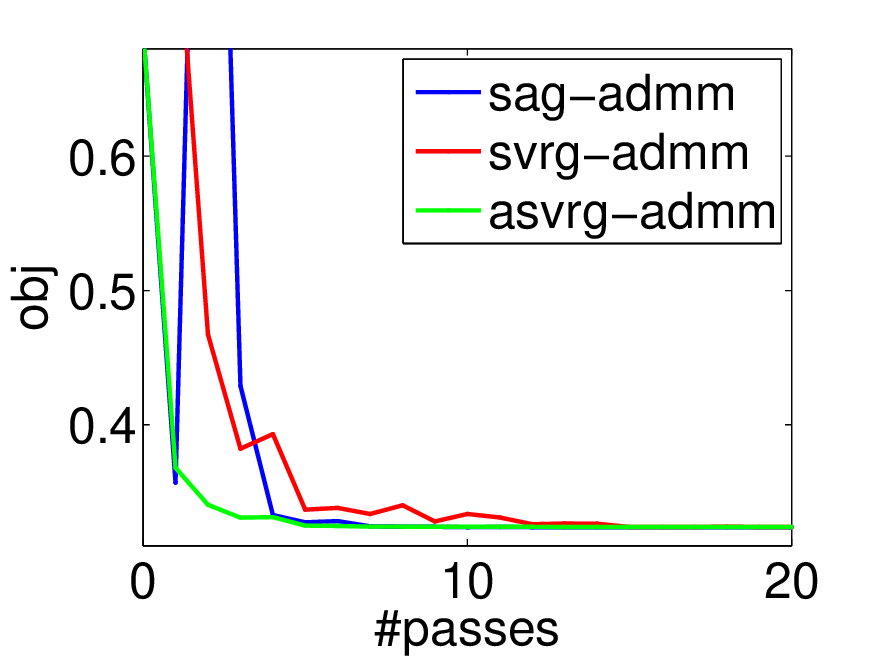}}
\subfigure[w8a]{\includegraphics[width = 40mm]{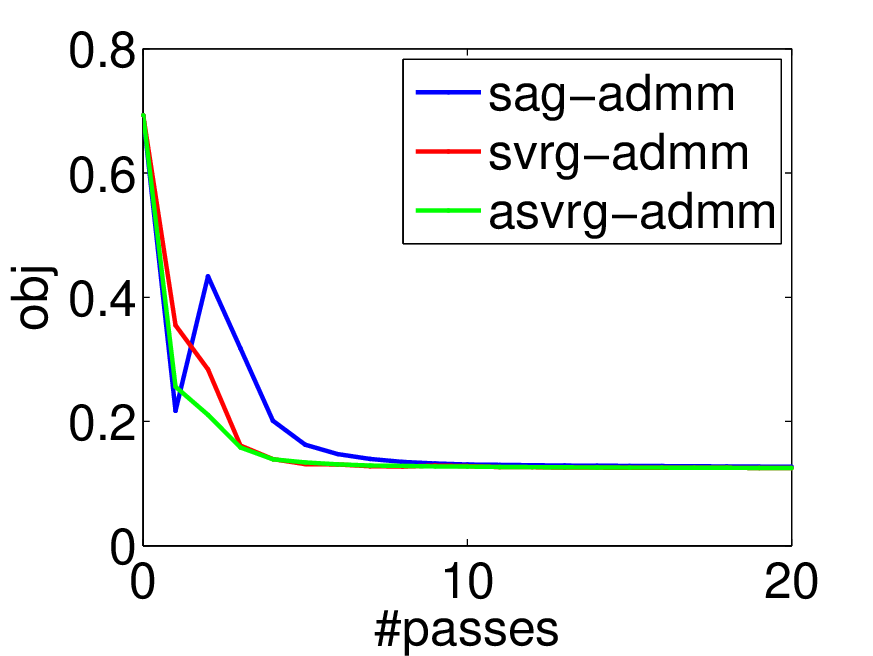}}
\subfigure[ijcnn1]{\includegraphics[width = 40mm]{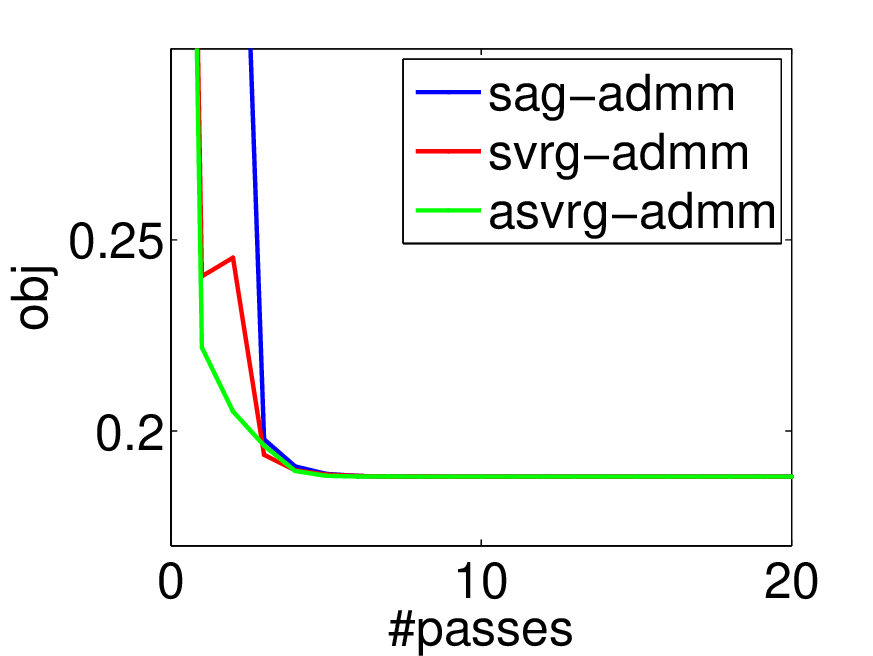}}
\subfigure[covertype]{\includegraphics[width = 40mm]{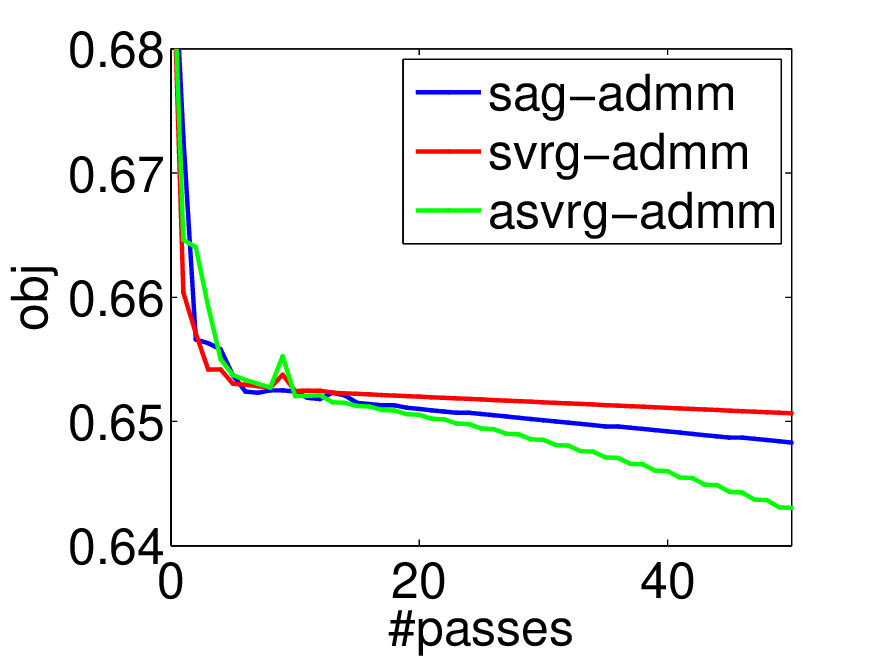}}
\end{center}
   \caption{Performance over effective passes of data for convex $f$}
\label{fig:error}
\end{figure*}

Besides, we have to mention that the parameter update rules we use here are different from all other accelerated stochastic methods with similar procedure \clr{\cite{allen2016katyusha,hien2016accelerated}},
where they fix the weight parameter $\alpha_{3}$ of $\tilde{x}_{s-1}$ to be a constant, and only change
$\alpha_{1,s}$ and $\alpha_{2,s}$.

Since $\eta_{s} = \bar{L}\alpha_{2,s} + \chi \|A\|_2^2\theta_{s}$ and $\theta_s = \beta_1 \alpha_{2,s}$,
the matrix inverse in x update rule (\ref{xupdate:non}) is
\[
(\eta_s I + \theta_s A^T A)^{-1} = \alpha_{2,s}^{-1}(\bar{L} I + \beta_1 A^TA)^{-1}.
\]
Then we can pre-compute $(\bar{L} I + \beta_1 A^TA)^{-1}$ and store it in order the decrease the computation.
According to the update rules of $\theta_s,\eta_s,\rho_s$,
we gradually decrease $\theta_s$ and $\eta_s$,
and increase $\rho_s$.
This means we increases the step size of the x-update and
the $\lambda$-update.

\section{Experimental Results}
In this section, we give the experimental results of the proposed algorithm.
We consider the generalized lasso \clr{\cite{tibshirani2011solution}}.
More specifically, we solve the following optimization problem:
\[
\min_{x,y} \frac{1}{n}\sum_{i=1}^n f_i(x) + \nu \|Fx\|_1,
\]
where the  penalty matrix  $F$  gives information about the underlying sparsity pattern of $x$.
While proximal methods have been used to solve this problem ( \clr{\cite{liu2010efficient}}, \clr{\cite{barbero2011fast}}),
the existence of $F$ makes the underlying proximal step difficult to solve.
This problem can be rewritten as
\[
\min_{x,y} \frac{1}{n}\sum_{i=1}^n f_i(x) + \nu \|y\|_1,\quad  s.t. \quad Fx -y =0.
\]
It can be solved efficiently by ADMM methods by dealing with $x$ and $y$ separately.
In our experiments, we focus on the graph guided fused lasso \clr{\cite{kim2009multivariate}},
where $F =[G;I]$ is constructed based on  a graph and $G$ denotes the sparsity pattern of graph.
We use SVRG-ADMM\clr{\cite{zheng2016fast} }and SAG-ADMM \clr{\cite{zhong2014fast}} as our baseline.
Since we can linearize $\|Fx - y\|^2_2$ in the augmented lagrangian in all these three method ,
here in our implementation we use this technique to reduce the computation complexity.

We conduct experiments on four real-world datasets from web machine learning repositories.
All the datasets were downloaded from the LIBSVM website, and their basic information are listed in Table 1.
We repeat all the experiments for 10 times  and report the average results.
For the parameters in ASVRG-ADMM,
we initialize $\alpha_{2,1} = 2/3$ , $\alpha_{3,1} = 1/10$,
and all other parameters are set as in Theorem \ref{them:main}.
For parameters in SAG-ADMM and SVRG-ADMM,
we all choose them as indicated in their papers,
and SAG-ADMM is  initialized by running SADMM for n iteration.
All the methods are implemented in Matlab,
and all the experiments are performed on a Windows server with two Intel Xeon E5-2690 CPU and 128GB memory.

If the algorithm uses one full gradient or n stochastic gradients, we call it uses one effective pass of data.
We report the objective value over effective passes of data.
We can see that  in all the four datasets,
our ASVRG-ADMM algorithm outperforms the other two algorithms.
On the first three data set, the smoothness constant is small and all the  three method converges after 20 effective passes of data,
the ASVRG-ADMM method converges faster then the other two algorithms.
When $L_f$ is large(on the covtype dataset),
all the three algorithms do not converge in 50 passes of data,
but the objective of ASVRG-ADMM decreases much faster than the other two algorithms.

\section{Conclusion}
In this paper, we combine the variance reduction technique and the acceleration technique in stochastic ADMM together,
we devised a new stochastic method with faster convergence rate,
especially in term of the smoothness constant of $f(x)$.
It is of great significance since big dataset often means a large smoothness constant.
Our experimental results also validate the effectiveness of our algorithm.
\bibliographystyle{plainnat}
\bibliography{asadmm}

\section{Appendix}
\subsection{Proof of Lemma 2}

\begin{proof}
By the smoothness of $f(\cdot)$,
we have
\[
f(x_{t,s}^{ag}) \le f(x_{t,s}^{md}) + \langle \nabla f(x_{t,s}^{md}),x_{t,s}^{ag}-x_{t,s}^{md}\rangle + \frac{L_f}{2} \|x_{t,s}^{ag}-x_{t,s}^{md}\|^2_2.
\]
According to line 7 and line 11 in Algorithm~\ref{alg:asvrgadmm},
we have $x_{t,s}^{ag} - x_{t,s}^{md} = \alpha_{2,s}(x_{t,s}-x_{t-1,s})$.
Substituting this into the above inequality,
we have
\begin{eqnarray*}
f(x_{t,s}^{ag}) &\le& f(x_{t,s}^{md}) + \langle \nabla f(x_{t,s}^{md}),x_{t,s}^{ag}-x_{t,s}^{md}\rangle + \frac{L_f\alpha_{2,s}^2}{2} \|x_{t,s}-x_{t-1,s}\|^2_2.
\\
&\le&f(x_{t,s}^{md}) + \langle \nabla f(x_{t,s}^{md})-v_{t,s},x_{t,s}^{ag}-x_{t,s}^{md}\rangle
+ \frac{L_f\alpha_{2,s}^2}{2} \|x_{t,s}-x_{t-1,s}\|^2_2+
\langle v_{t,s}, x_{t,s}^{ag}-x_{t,s}^{md}\rangle
\\
&\le&f(x_{t,s}^{md}) + \frac{\alpha_{3,s}}{2L_Q}\|\nabla f(x_{t,s}^{md})-v_{t,s}\|^2_2+\frac{L_Q}{2\alpha_{3,s}}\|x_{t,s}^{ag}-x_{t,s}^{md}\|^2_2
+ \frac{L_f\alpha_{2,s}^2}{2} \|x_{t,s}-x_{t-1,s}\|^2_2
\\
&&+\langle v_{t,s}, \alpha_{1,s}x_{t-1,s}^{ag}+\alpha_{2,s}x_{t,s}+\alpha_{3,s} \tilde{x}_{s-1}-x_{t,s}^{md}\rangle
\\
&\le&
\alpha_{1,s}(f(x_{t,s}^{md})+\langle v_{t,s},x_{t-1,s}^{ag}- x_{t,s}^{md}\rangle)
+\alpha_{2,s}(f(x_{t,s}^{md})+\langle v_{t,s},x- x_{t,s}^{md}\rangle)
\\
&&
+\alpha_{3,s}(f(x_{t,s}^{md})+\langle v_{t,s},\tilde{x}_{s-1}- x_{t,s}^{md}\rangle+\frac{1}{2L_Q}\|\nabla f(x_{t,s}^{md})-v_{t,s}\|^2)
\\
&&
\alpha_{2,s}\langle v_{t,s},x_{t,s}-x \rangle + \frac{\alpha_{2,s}^2({L_Q}/{\alpha_{3,s}}+L_f)}{2}\|x_{t,s}-x_{t-1,s}\|^2_2
\end{eqnarray*}
In the third inequality,
we use the Cauchy-Schwartz inequality $2\langle a,b\rangle \le \|a\|^2 + \|b\|^2$.
According to the optimality of $x_{t,s}$,
we have
\[
v_{t,s} + \chi \theta_{s} A^T(Ax_{t-1,s}+By_{t-1,s}-c)+
(1-\chi)\theta_{s}A^T(Ax_{t,s}+B y_{t-1,s} -c) +A^T\lambda_{t-1,s} + \eta_{s} (x_{t,s}-x_{t-1,s})=0.
\]
If we denote $\hat{x}_{t,s} = \chi x_{t,s}+(1-\chi)x_{t-1,s}$,
we have
\begin{eqnarray*}
v_{t,s} &=& \eta_{s} (x_{t-1,s} - x_{t,s}) - A^T (\theta_{s}(A \hat{x}_{t,s} + B y_{t-1,s} -c )+ \lambda_{t-1,s})
\\
&=& \eta_{s} (x_{t-1,s} - x_{t,s})  - A^T(\frac{\theta_{s}}{\rho_{s}}(\lambda_{t,s}-\lambda_{t-1,s})
+\theta_{s}A(\hat{x}_{t,s}-x_{t,s})+\theta_{s}B(y_{t-1,s}-y_{t,s}))
\end{eqnarray*}
Since $\bar{L} \ge {L_Q}/{\alpha_{3,s}}+L_f$,
$\E (v_{t,s}) = \nabla f(x_{t,s}^{md})$
and $x_{t-1,s}^{ag}, x_{t,s}^{md},\tilde{x}_{s-1}$  is independent of $v_{t,s}$,
we have
\begin{eqnarray}\label{xrelation}
\E f(x_{t,s}^{ag})
&\le&
\alpha_{1,s}(f(x_{t,s}^{md})+\langle \nabla f(x_{t,s}^{md}),x_{t-1,s}^{ag}- x_{t,s}^{md}\rangle)
+\alpha_{2,s}(f(x_{t,s}^{md})+\langle \nabla f(x_{t,s}^{md}),x- x_{t,s}^{md}\rangle)
\nonumber\\
&&
+\alpha_{3,s}(f(x_{t,s}^{md})+\langle \nabla f(x_{t,s}^{md}),\tilde{x}_{s-1}- x_{t,s}^{md}\rangle+\frac{1}{2L_Q}\|\nabla f(x_{t,s}^{md})-v_{t,s}\|^2)+
\nonumber\\
&&\alpha_{2,s}\langle \eta_{s} (x_{t-1,s} - x_{t,s}),x_{t,s}-x \rangle  - (\theta_{s}(A \hat{x}_{t,s} + B y_{t-1,s} -c) + \lambda_{t-1,s}),A(x_{t,s}-x) \rangle +
\nonumber\\
&&\frac{\alpha_{2,s}^2\bar{L}}{2}\|x_{t,s}-x_{t-1,s}\|^2_2
\nonumber\\
&\le&\alpha_{1,s}f(x_{t-1,s}^{ag})+\alpha_{2,s}f(x)+
\alpha_{3,s}f(\tilde{x}_{s-1})+
\alpha_{2,s}\langle \eta_{s} (x_{t-1,s} - x_{t,s}) ,x_{t,s}-x \rangle-
\nonumber\\&&\alpha_{2,s}\langle \frac{\theta_{s}}{\rho_{s}}(\lambda_{t,s}-\lambda_{t-1,s})+\lambda_{t-1,s}
+\theta_{s}A(\hat{x}_{t,s}-x_{t,s})+\theta_{s}B(y_{t-1,s}-y_{t,s}),A(x_{t,s}-x) \rangle
\nonumber\\
&&+\frac{\alpha_{2,s}^2\bar{L}}{2}\|x_{t,s}-x_{t-1,s}\|^2_2
\end{eqnarray}
The second inequality is due to Lemma 1 and the convexity of $f(x)$.

According to the optimality of Line 12 in Algorithm~\ref{alg:asvrgadmm} and the convexity of $g(y)$,
we have
\begin{eqnarray}\label{yrelation}
g(y_{t,s})-g(y)&\le&\langle \partial g(y_{t,s}),y_{t,s}-y \rangle \nonumber\\
 &=&\langle \theta_{s} (Ax_{t,s}+By_{t,s}-c)+\lambda_{t-1,s},B(y-y_{t,s})\rangle.
\end{eqnarray}

According to the definition of gap function, we have
\begin{eqnarray}\label{Qrelation}
&&\E Q(w;w_{t,s}^{ag})-\alpha_{1,s}Q(w;w_{t-1,s}^{ag})-\alpha_{3,s}Q(w;\tilde{w}_{s-1})
\nonumber\\
&=&[f(x_{t,s}^{ag})+g(y_{t,s}^{ag})+\langle \lambda,Ax_{t,s}^{ag} + By_{t,s}^{ag} -c\rangle]
-
[f(x)+g(y)+\langle \lambda_{t,s}^{ag},Ax + By -c\rangle]
\nonumber\\
&&-\alpha_{1,s}([f(x_{t-1,s}^{ag})+g(y_{t-1,s}^{ag})+\langle \lambda,Ax_{t-1,s}^{ag} + By_{t-1,s}^{ag} -c\rangle]
-
[f(x)+g(y)+\langle \lambda_{t-1,s}^{ag},Ax + By -c\rangle])
\nonumber\\
&&-\alpha_{3,s}(
[f(\tilde{x}_{s-1})+g(\tilde{y}_{s-1})+\langle \lambda,A\tilde{x}_{s-1} + B\tilde{y}_{s-1} -c\rangle]
-
[f(x)+g(y)+\langle \tilde\lambda_{s-1},Ax + By -c\rangle])
\nonumber\\
&=&[f(x_{t,s}^{ag})-\alpha_{1,s}f(x_{t-1,s}^{ag})-\alpha_{2,s}f(x)
-\alpha_{3,s} f(\tilde{x}_{s-1})]+[g(y_{t,s}^{ag})-\alpha_{1,s}g(y_{t-1,s}^{ag})-\alpha_{2,s}g(y)
\nonumber\\
&&-\alpha_{3,s} g(\tilde{y}_{s-1})]
+\alpha_{2,s}\langle \lambda , Ax_{s,t}+By_{s,t}-c\rangle -\alpha_{2,s}\langle \lambda_{t,s} , Ax+By-c\rangle
\nonumber\\
&\le&
[f(x_{t,s}^{ag})-\alpha_{1,s}f(x_{t-1,s}^{ag})-\alpha_{2,s}f(x)
-\alpha_{3,s} f(\tilde{x}_{s-1})]+\alpha_{2,s}[g(y_{t,s})-g(y)]
\nonumber\\
&&+\alpha_{2,s}\langle \lambda , Ax_{s,t}+By_{s,t}-c\rangle -\alpha_{2,s}\langle \lambda_{t,s} , Ax+By-c\rangle
\end{eqnarray}
Combining (\ref{xrelation}),~(\ref{yrelation}) and~(\ref{Qrelation}),
we have
\begin{eqnarray}\label{iterQ}
&&\E~Q(w;w_{t,s}^{ag})-\alpha_{1,s,t}Q(w;w_{t-1,s}^{ag})-\alpha_{3,s,t}Q(w;\tilde{w}_{s-1})
\nonumber\\
&\le&\alpha_{2,s}[\langle \eta_{s} (x_{t-1,s} - x_{t,s}) ,x_{t,s}-x \rangle-
\nonumber\\
&&\langle \frac{\theta_{s}}{\rho_{s}}(\lambda_{t,s}-\lambda_{t-1,s})+\lambda_{t-1,s}
+\theta_{s}A(\hat{x}_{t,s}-x_{t,s})+\theta_{s}B(y_{t-1,s}-y_{t,s}),A(x_{t,s}-x) \rangle
\nonumber\\
&&+\frac{\alpha_{2,s}\bar{L}}{2}\|x_{t,s}-x_{t-1,s}\|^2_2
+\langle \theta_{s} (Ax_{t,s}+By_{t,s}-c)+\lambda_{t-1,s},B(y-y_{t,s})\rangle
\nonumber\\
&&+\langle \lambda , Ax_{s,t}+By_{s,t}-c\rangle -\langle \lambda_{t,s} , Ax+By-c\rangle]
\nonumber\\
&\le&\alpha_{2,s}[\langle \eta_{s} (x_{t-1,s} - x_{t,s}) ,x_{t,s}-x \rangle-
\nonumber\\
&&\langle \frac{\theta_{s}}{\rho_{s}}(\lambda_{t,s}-\lambda_{t-1,s})+\lambda_{t-1,s}
+\theta_{s}A(\hat{x}_{t,s}-x_{t,s})+\theta_{s}B(y_{t-1,s}-y_{t,s}),A(x_{t,s}-x) \rangle
\nonumber\\
&&+\frac{\alpha_{2,s}\bar{L}}{2}\|x_{t,s}-x_{t-1,s}\|^2_2
+\langle \theta_{s} (Ax_{t,s}+By_{t,s}-c)+\lambda_{t-1,s},B(y-y_{t,s})\rangle
\nonumber\\
&&+\langle \lambda - \lambda_{t,s} , Ax_{s,t}+By_{s,t}-c\rangle +\langle \lambda_{t,s} , Ax_{t,s}-Ax\rangle+\langle \lambda_{t,s},By_{t,s}-By\rangle]
\nonumber\\
&\le&\alpha_{2,s}[\langle \eta_{s} (x_{t-1,s} - x_{t,s}) ,x_{t,s}-x \rangle+\langle \lambda - \lambda_{t,s} , \frac{1}{\rho_{s}}{(\lambda_{t,s}-\lambda_{t-1,s})}\rangle-
\nonumber\\
&&\langle (\frac{\theta_{s}}{\rho_{s}}-1)(\lambda_{t,s}-\lambda_{t-1,s}),A(x_{t,s}-x)
\rangle+\langle (\frac{\theta_{s}}{\rho_{s}}-1)(\lambda_{t,s}-\lambda_{t-1,s}),B(y-y_{t,s})\rangle
\nonumber\\
&&-\theta_{s}\langle A(\hat{x}_{t,s}-x_{t,s}),A(x_{t,s}-x)
\rangle
-\theta_{s}\langle B(y_{t-1,s}-y_{t,s}),A(x_{t,s}-x) \rangle
\nonumber\\
&&+\frac{\alpha_{2,s}\bar{L}}{2}\|x_{t,s}-x_{t-1,s}\|^2_2
\nonumber\\
\end{eqnarray}

Since $2\langle a,b\rangle = \|a+b\|^2-\|a\|^2-\|b\|^2$,we have
\begin{equation}\label{iterx}
\langle x_{t-1,s}-x_{t,s},x_{t,s}-x\rangle =\frac{1}{2}( \|x_{t-1,s}-x\|^2 -\|x_{t,s}-x\|^2 - \|x_{t-1,s}-x_{t,s}\|^2)
\end{equation}
\begin{equation}\label{iterlambda}
\langle \lambda-\lambda_{t,s},\lambda_{t,s}-\lambda_{t-1,s} \rangle = \frac{1}{2}(\|\lambda_{t-1,s}-\lambda\|^2 - \|\lambda_{t,s}-\lambda\|^2 - \|\lambda_{t,s}-\lambda_{t-1,s}\|^2).
\end{equation}
Since $B(y-y_{t,s}) = (Ax + By -c )-\frac{1}{\rho_{s}}(\lambda_{t,s}-\lambda_{t-1,s})+A(x_{t,s}-x)$,
we have
\begin{eqnarray}\label{itertheta}
&&-\langle (\frac{\theta_{s}}{\rho_{s}}-1)(\lambda_{t,s}-\lambda_{t-1,s}),A(x_{t,s}-x)
\rangle+\langle (\frac{\theta_{s}}{\rho_{s}}-1)(\lambda_{t,s}-\lambda_{t-1,s}),B(y-y_{t,s})\rangle
\nonumber\\
&=&-(\frac{\theta_{s}}{\rho_{s}}-1)(\lambda_{t,s}-\lambda_{t-1,s}),A(x_{t,s}-x)
\rangle+\langle (\frac{\theta_{s}}{\rho_{s}}-1)(\lambda_{t,s}-\lambda_{t-1,s}),(Ax + By -c )\rangle
\nonumber\\
&&+\langle (\frac{\theta_{s}}{\rho_{s}}-1)(\lambda_{t,s}-\lambda_{t-1,s}),-\frac{1}{\rho_{t,s}}(\lambda_{t,s}-\lambda_{t-1,s})+A(x_{t,s}-x)\rangle
\nonumber\\
&=&-\frac{\theta_{s}-\rho_{s}}{\rho_{s}}\|\lambda_{t,s}-\lambda_{t-1,s}\|^2+
\langle (\frac{\theta_{s}}{\rho_{s}}-1)(\lambda_{t,s}-\lambda_{t-1,s}),(Ax + By -c )
\nonumber\\
\end{eqnarray}
We also have
\begin{eqnarray}\label{iterrho}
&&-\theta_{s}\langle A(\hat{x}_{t,s}-x_{t,s}),A(x_{t,s}-x)
\rangle
+\theta_{s}\langle B(y_{t,s}-y_{t-1,s}),A(x_{t,s}-x) \rangle
\nonumber\\
&=&\frac{\chi\theta_{s}}{2}(\|A(x_{t-1,s}-x_{t,s})\|^2+\|A(x_{t-1,s}-x)\|^2
-\|A(x_{t,s}-x)\|^2)
\nonumber\\
&&+\frac{\theta_{s}}{2}(\|Ax_{t,s}+By_{t,s}-c\|^2-\|Ax+By_{t,s}-c\|^2                                                     \nonumber\\
&&+\|Ax+By_{t-1,s}-c\|^2-\|Ax_{t,s}+By_{t-1,s}-c\|^2)
\nonumber\\
&\le&\frac{\chi\theta_{s}}{2}(\|A(x_{t-1,s}-x)\|^2
-\|A(x_{t,s}-x)\|^2)+\frac{\chi\theta_{s}\|A\|^2_2}{2}\|(x_{t-1,s}-x_{t,s})\|^2
\nonumber\\
&&+\frac{\theta_{s}}{2\rho_{s}^2}\|\lambda_{t,s}-\lambda_{t-1,s}\|^2+\frac{\theta_{s}}{2}(\|Ax+By_{t-1,s}-c\|^2-\|Ax+By_{t,s}-c\|^2                                                     )\nonumber\\
&&-\frac{\theta_{s}}{2}\|Ax_{t,s}+By_{t-1,s}-c\|^2
\nonumber\\
\end{eqnarray}

Substituting (\ref{iterx}), (\ref{iterlambda}), (\ref{itertheta}) and (\ref{iterrho}) into
(\ref{iterQ}), we have
\begin{eqnarray}
&&\E Q(w;w_{t,s}^{ag})-\alpha_{1,t,s}Q(w;w_{t-1,s}^{ag})-\alpha_{3,s,t}Q(w;\tilde{w}_{s-1})
\nonumber\\
&\le&\alpha_{2,s}[\frac{\eta_{s}}{2}(\|x_{t-1,s}-x\|^2-\|x_{t,s}-x\|^2)+\frac{1}{2\rho_{s}}(\|\lambda_{t-1,s}-\lambda\|^2-\|\lambda_{t,s}-\lambda)\|^2
\nonumber\\
&&+\frac{\chi\theta_{s}}{2}(\|A(x_{t,s}-x)\|^2-\|A(x_{t-1,s}-x)\|^2)+\frac{\theta_{s}}{2}(\|Ax+By_{t-1,s}-c\|^2-
\nonumber\\
&&\|Ax+By_{t,s}-c\|^2)+\langle (\frac{\theta_{s}}{\rho_{s}}-1)(\lambda_{t,s}-\lambda_{t-1,s}),(Ax + By -c )\rangle
\nonumber\\
&&+\frac{\rho_{s}-\theta_{s}}{2\rho_{s}^2}\|\lambda_{t,s}-\lambda_{t-1,s}\|^2
-\frac{\theta_{s}}{2}\|Ax_{t,s}+By_{t-1,s}-c\|^2\nonumber\\
&&-\frac{\eta_{s}-\bar{L}\alpha_{2,t}-\chi\theta_{s}\|A\|^2_2}{2}
\|x_{t,s}-x_{t-1,s}\|]
\end{eqnarray}
Since $Ax^*+By^*-c=0$,
if we choose $\theta_{s}\ge\rho_{s}$ and $\eta_{s}\ge\bar{L}\alpha_{2,s}+\chi\theta_{s}\|A\|^2_2$,
we have
\begin{eqnarray}\label{Qloop}
&&\E Q(x^*,y^*,\lambda;w_{t,s}^{ag})-\alpha_{1,s}Q(x^*,y^*,\lambda;w_{t-1,s}^{ag})-\alpha_{3,t}Q(x^*,y^*,\lambda;\tilde{w}_{s-1})
\nonumber\\
&\le&\alpha_{2,t}[\frac{\eta_{s}}{2}(\|x_{t-1,s}-x^*\|^2-\|x_{t,s}-x^*\|^2)+\frac{1}{2\rho_{t,s}}(\|\lambda_{t-1,s}-\lambda\|^2-\|\lambda_{t,s}-\lambda\|^2)
\nonumber\\
&&+\frac{\chi\theta_{t,s}}{2}(\|A(x_{t,s}-x^*)\|^2-\|A(x_{t-1,s}-x^*)\|^2)+\frac{\theta_{t,s}}{2}(\|Ax^*+By_{t-1,s}-c\|^2-
\nonumber\\
&&\|Ax^*+By_{t,s}-c\|^2)]
\end{eqnarray}
\end{proof}

\subsection{Proof of Lemma 3}
\begin{proof}
In the $s-th$ outer iteration, if we make $\theta_{s}= \beta_1\alpha_{2,s}$, $\rho_{s}=\frac{\beta_2}{\alpha_{2,s}}$ ,$\eta_{s}=(\bar{L}+\chi\beta_1\|A\|_2^2)\alpha_{2,s}$
then according to Lemma 2 we have
\begin{eqnarray}
&&\frac{1}{\alpha_{2,s}^2}\E Q(x^*,y^*,\lambda;w_{t,s}^{ag})-\frac{\alpha_{1,s}}{\alpha^2_{2,s}}Q(x^*,y^*,\lambda;w_{t-1,s}^{ag})-\frac{\alpha_{3,s}}{\alpha_{2,s}^2}Q(x^*,y^*,\lambda;\tilde{w}_{s-1})
\nonumber\\
&\le&\frac{\bar{L}+\chi\beta_1\|A\|_2^2}{2}(\|x_{t-1,s}-x^*\|^2-\|x_{t,s}-x^*\|^2)+\frac{\beta_2}{2}(\|\lambda_{t-1,s}-\lambda\|^2-\|\lambda_{t,s}-\lambda\|^2)
\nonumber\\
&&+\frac{\chi\beta_1}{2}(\|A(x_{t,s}-x^*)\|^2-\|A(x_{t-1,s}-x^*)\|^2)+\frac{\beta_1}{2}(\|Ax^*+By_{t-1,s}-c\|^2-
\nonumber\\
&&\|Ax^*+By_{t,s}-c\|^2)]
\end{eqnarray}

Adding up t from $1$ to $m$ in the $s-th$ outer iteration,
we have
\begin{eqnarray}
&&\frac{1}{\alpha_{2,s}^2}\E Q(x^*,y^*,\lambda;w_{m,s}^{ag})+\sum_{t=1}^{m-1}\frac{1-\alpha_{1,s}}{\alpha^2_{2,s}}\E Q(x^*,y^*,\lambda;w_{t,s}^{ag})
\nonumber\\
&\le&\frac{\alpha_{1,s}}{\alpha_{2,s}^2}\E Q(x^*,y^*,\lambda;w_{0,s}^{ag})+\frac{\alpha_{3,s}m}{\alpha_{2,s}^2}\E Q(x^*,y^*,\lambda;\tilde{w}_{s-1})
\nonumber\\
&&+\frac{\bar{L}+\chi\beta_1\|A\|_2^2}{2}(\|x_{0,s}-x^*\|^2-\|x_{m,s}-x^*\|^2)+\frac{1}{2\beta_2}(\|\lambda_{0,s}-\lambda\|^2-\|\lambda_{m,s}-\lambda\|^2)
\nonumber\\
&&+\frac{\chi\beta_1}{2}(\|A(x_{m,s}-x^*)\|^2-\|A(x_{0,s}-x^*)\|^2)+\frac{\beta_1}{2}(\|Ax^*+By_{0,s}-c\|^2-
\nonumber\\
&&\|Ax^*+By_{m,s}-c\|^2)]
\end{eqnarray}

If $\frac{1}{\alpha_{2,s}^2}=\frac{1-\alpha_{2,s+1}}{\alpha_{2,s+1}^2}$ and $\frac{1-\alpha_{1,s}}{\alpha^2_{2,s}}=\frac{\alpha_{3,s+1}}{\alpha_{2,s+1}^2}$,
we have
\begin{eqnarray}
&&\frac{1-\alpha_{2,s+1}}{\alpha_{2,s+1}^2}\E Q(x^*,y^*,\lambda;w_{m,s}^{ag})+\sum_{t=1}^{m-1}\frac{\alpha_{3,s+1}}{\alpha_{2,s+1}^2}\E Q(x^*,y^*,\lambda;w_{t,s}^{ag})
\nonumber\\
&\le&\frac{\alpha_{1,s}}{\alpha_{2,s}^2}\E Q(x^*,y^*,\lambda;w_{0,s}^{ag})+\frac{\alpha_{3,s}m}{\alpha_{2,s}^2}\E Q(x^*,y^*,\lambda;\tilde{w}_{s-1})
\nonumber\\
&&+\frac{\bar{L}+\chi\beta_1\|A\|_2^2}{2}(\|x_{0,s}-x^*\|^2-\|x_{m,s}-x^*\|^2)+\frac{1}{2\beta_2}(\|\lambda_{0,s}-\lambda\|^2-\|\lambda_{m,s}-\lambda\|^2)
\nonumber\\
&&+\frac{\chi\beta_1}{2}(\|A(x_{m,s}-x^*)\|^2-\|A(x_{0,s}-x^*)\|^2)+\frac{\beta_1}{2}(\|Ax^*+By_{0,s}-c\|^2-
\nonumber\\
&&\|Ax^*+By_{m,s}-c\|^2)]
\end{eqnarray}

According to the convexity of $f(x)$ and $g(y)$ and the linearity of $\lambda (Ax+By-c)$ (w.r.t. $x \& y$),
we have
\begin{eqnarray*}
Q(x^*,y^*,\lambda;\tilde{w}_{s})  &=&f(\tilde{x}_{s})-f(x^*)+g(\tilde{y}_{s})-g(y^*)+
\langle\lambda,A\tilde{x}_s+B\tilde{y}_s-c\rangle
\\
&\le&\frac{1}{m}\sum_{t=1}^m (f({x}_{t,s}^{ag})-f(x^*)+g(y_{t,s}^{ag})-g(y^*)+
\langle\lambda,A{x}_{t,s}^{ag}+By_{t,s}^{ag}-c\rangle)
\\
&=&\frac{1}{m}\sum_{t=1}^mQ(x^*,y^*,\lambda;w_{t,s}^{ag})
\end{eqnarray*}

Since $x_{0,s} = x_{m,s-1}$, $y_{0,s} = y_{m,s-1}$ and $\lambda_{0,s} = \lambda_{m,s-1}$,
we have
\begin{eqnarray}
&&\frac{\alpha_{1,s+1}}{\alpha_{2,s+1}^2}\E Q(x^*,y^*,\lambda;w_{m,s}^{ag})+\frac{\alpha_{3,s+1}m}{\alpha_{2,s+1}^2}\E Q(x^*,y^*,\lambda;\tilde{w}_{s})
\nonumber\\
&\le&\frac{\alpha_{1,s}}{\alpha_{2,s}^2}\E Q(x^*,y^*,\lambda;w_{m,s-1}^{ag})+\frac{\alpha_{3,s}m}{\alpha_{2,s}^2}\E Q(x^*,y^*,\lambda;\tilde{w}_{s-1})
\nonumber\\
&&+\frac{\bar{L}+\chi\beta_1\|A\|_2^2}{2}(\|x_{m,s-1}-x^*\|^2-\|x_{m,s}-x^*\|^2)+\frac{1}{2\beta_2}(\|\lambda_{m,s-1}-\lambda\|^2-\|\lambda_{m,s}-\lambda\|^2)
\nonumber\\
&&+\frac{\chi\beta_1}{2}(\|A(x_{m,s}-x^*)\|^2-\|A(x_{m,s-1}-x^*)\|^2)+\frac{\beta_1}{2}(\|Ax^*+By_{m,s-1}-c\|^2-
\nonumber\\
&&\|Ax^*+By_{m,s}-c\|^2)]
\end{eqnarray}

Summing s from 1 to N, we have
\begin{eqnarray}
&&\frac{\alpha_{1,N+1}}{\alpha_{2,N+1}^2}\E Q(x^*,y^*,\lambda;w_{m,N}^{ag})+\frac{\alpha_{3,N+1}m}{\alpha_{2,N+1}^2} \E Q(x^*,y^*,\lambda;\tilde{w}_{N})
\nonumber\\
&\le&\frac{\alpha_{1,1}}{\alpha_{2,1}^2}Q(x^*,y^*,\lambda;w_{m,0}^{ag})+\frac{\alpha_{3,1}m}{\alpha_{2,1}^2}Q(x^*,y^*,\lambda;\tilde{w}_{0})
\nonumber\\
&&+\frac{\bar{L}+\chi\beta_1\|A\|_2^2}{2}(\|x_{m,0}-x^*\|^2-\|x_{m,N}-x^*\|^2)+\frac{1}{2\beta_2}(\|\lambda_{m,0}-\lambda\|^2-\|\lambda_{m,N}-\lambda\|^2)
\nonumber\\
&&+\frac{\chi\beta_1}{2}(\|A(x_{m,N}-x^*)\|^2-\|A(x_{m,0}-x^*)\|^2)+\frac{\beta_1}{2}(\|Ax^*+By_{m,0}-c\|^2-
\nonumber\\
&&\|Ax^*+By_{m,N}-c\|^2)]
\end{eqnarray}

According to the convexity of $Q(x^*,y^*,\lambda;\cdot)$, if we set $\hat{w}_N = \frac{\alpha_{1,N+1}}{\alpha_{1,N+1}+\alpha_{3,N+1}m}w_{m,N}^{ag}+\frac{\alpha_{3,N+1}m}{\alpha_{1,N+1}+\alpha_{3,N+1}m}\tilde{w}_{N}$
,
we have
\begin{eqnarray}\label{finalrelation}
&&\frac{\alpha_{1,N+1}+\alpha_{3,N+1}m}{\alpha_{2,N+1}^2}\E Q(x^*,y^*,\lambda;\hat{w}_{N})
\nonumber\\
&\le&\frac{\alpha_{1,1}}{\alpha_{2,1}^2} Q(x^*,y^*,\lambda;w_{m,0}^{ag})+\frac{\alpha_{3,1}m}{\alpha_{2,1}^2}Q(x^*,y^*,\lambda;\tilde{w}_{0})
+\frac{\bar{L}+\chi\beta_1\|A\|_2^2}{2}\|x_{m,0}-x^*\|^2\nonumber\\
&&+\frac{1}{2\beta_2}\|\lambda_{m,0}-\lambda\|^2
+\frac{\chi\beta_1}{2}\|A(x_{m,N}-x^*)\|^2+\frac{\beta_1}{2}\|Ax^*+By_{m,0}-c\|^2
\end{eqnarray}

The above inequality is true for all $\lambda$,
hence it also holds in the ball $\B_0 = \{\lambda:\|\lambda\|_2 \le \gamma\}$,
it follows that
\begin{eqnarray}
\max_{\lambda\in\B_0} Q(x^*,y^*,\lambda;\hat{x},\hat{y},\hat{\lambda})
&=& \max_{\lambda\in\B_0}\{f(\hat{x})+g(\hat{y}) - f(x^*) - g(y^*)+\langle \lambda,A\hat{x}+B\hat{y}-c\rangle\}
\nonumber\\
&=&f(\hat{x})+g(\hat{y}) - f(x^*) - g(y^*)+\gamma\|A\hat{x}+B\hat{y}-c\|^2
\end{eqnarray}
If we make both side of inequality (\ref{finalrelation}) the max of $\lambda \in \B_0$, then we have
\begin{eqnarray}
&&\E [f(\hat{x}_{N})+g(\hat{y}_{N}) - f(x^*) - g(y^*)+\gamma\|A\hat{x}_N+B\hat{y}_N-c\|^2]
\nonumber\\
&\le&\frac{\alpha_{2,N+1}^2}{\alpha_{1,N+1}+m\alpha_{3,N+1}}[\frac{1}{\alpha_{2,1}^2}(f(\tilde{x}_0)+g(\tilde{y}_0) - f(x^*) - g(y^*))
+\frac{\bar{L}+\chi\beta_1\|A\|_2^2}{2}\|x_{m,0}-x^*\|^2\nonumber\\
&&+\max_{\lambda\in\B_0}\frac{1}{2\beta_2}\|\lambda_{m,0}-\lambda\|^2
+\frac{\chi\beta_1}{2}\|A(x_{m,N}-x^*)\|^2+\frac{\beta_1}{2}\|By^*+By_{m,0}\|^2]
\nonumber \\
&\le&\frac{\alpha_{2,N+1}^2}{\alpha_{1,N+1}+m\alpha_{3,N+1}}[\frac{1+m\alpha_{3,1}}{\alpha_{2,1}^2}(f(\tilde{x}_0)+g(\tilde{y}_0) - f(x^*) - g(y^*))
+\frac{\bar{L}+\chi\beta_1\|A\|_2^2}{2}D^2_{x^*}\nonumber\\
&&+\frac{1}{2\beta_2}\gamma^2
+\frac{\chi\beta_1}{2}D^2_{A,X}+\frac{\beta_1}{2}D^2_{y^*,B}]
\end{eqnarray},
where the first inequality is due to the initialization $A \tilde{x}_0 +B\tilde{y}_0  -c = 0$ and $Ax^* + By^* - c =0$.
\end{proof}

\end{document}

%% file: format.tex
\usepackage[hmarginratio=1:1,top=32mm,columnsep=20pt]{geometry}
\usepackage[font=it]{caption}
\usepackage{paralist}
\usepackage{multicol}

\usepackage{lettrine}

\usepackage{abstract}

\usepackage{titlesec}
\renewcommand\thesection{\Roman{section}}
\titleformat{\section}[block]{\large\scshape\centering}{\thesection.}{1em}{}

\usepackage{fancyhdr}
\usepackage{hyperref}

\usepackage{amsthm}
\usepackage{amsmath}
\usepackage{amssymb}

\makeatletter
\let\start@align@nopar\start@align
\let\start@gather@nopar\start@gather
\let\start@multline@nopar\start@multline
\long\def\start@align{\par\start@align@nopar}
\long\def\start@gather{\par\start@gather@nopar}
\long\def\start@multline{\par\start@multline@nopar}
\makeatother

\usepackage{epsfig}
\usepackage{epstopdf}

\usepackage{graphicx}
\usepackage{subfigure}

\usepackage[linesnumbered,lined,ruled]{algorithm2e}

\usepackage{xcolor}

\usepackage{comment}

\usepackage[square, comma, sort]{natbib}

\usepackage{lineno}
\usepackage{setspace}

\newcommand{\SimpleTitle}[3]
{
\title{ \vspace{-12mm}%
	\fontsize{24pt}{10pt}\selectfont
	\textbf{#1}
}	
\author{%
	\large
	#2 \\[2mm]
	\normalsize	{${}^1$Computer Science Institute, Zhejiang University}\\
	\normalsize	{${}^2$Department of Respiratory Diseases, The First Affiliated Hospital, College of Medicine}\\
	\normalsize{ Zhejiang University}\\
	\normalsize	{\{#3\}}
	\vspace{-0mm}
}
\date{\today}
}